\documentclass[11pt]{article}
\usepackage{amsmath, amssymb, amsthm, verbatim,enumerate,bbm}
\usepackage{indentfirst}
\usepackage{amsmath}
\usepackage{amsthm}
\usepackage{amsfonts}
\usepackage{amssymb}
\usepackage{verbatim}
\usepackage{graphicx}
\usepackage[margin=1.0in]{geometry}
\usepackage{url}
\usepackage{enumerate}
\usepackage[pdftex,bookmarks=true]{hyperref}

\title{Embedding large graphs into a random graph}

\author{ Asaf Ferber
\thanks{Department of Mathematics, Yale University, and Department of Mathematics, MIT. Emails:
asaf.ferber@yale.edu, and ferbera@mit.edu.}\and Kyle Luh \thanks {Department of Mathematics, Yale University. Email: kyle.luh@yale.edu.}
 \and
Oanh Nguyen \thanks{Department of Mathematics, Yale University. Email: oanh.nguyen@yale.edu.}
}

\date{\today}
\allowdisplaybreaks

\theoremstyle{plain}
\newtheorem{theorem}{Theorem}[section]
\newtheorem{lemma}[theorem]{Lemma}

\newtheorem{conjecture}[theorem]{Conjecture}
\newtheorem{problem}[theorem]{Problem}
\newtheorem{remark}[theorem]{Remark}
\newtheorem{definition}[theorem]{Definition}

\newcommand{\eps}{\varepsilon}

\renewcommand{\P}{\mathbb P}

\newcommand{\indeg}{\text{indeg}}

\title{Embedding Large Graphs into a Random Graph}

\begin{document}
\maketitle

\begin{abstract}
In this paper we consider the problem of embedding almost-spanning, bounded degree graphs in a random graph. In particular, let $\Delta\geq 5$, $\varepsilon > 0$ and let $H$ be a graph on $(1-\varepsilon)n$ vertices and with maximum degree $\Delta$. We show that a random graph $G_{n,p}$ with high probability contains a copy of $H$, provided that $p\gg (n^{-1}\log^{1/\Delta}n)^{2/(\Delta+1)}$. Our assumption on $p$ is optimal up to the $polylog$ factor.  We note that this $polylog$ term matches the conjectured threshold for the spanning case.
\end{abstract}

\section{Introduction}

Ever since its introduction by Erd{\H
o}s and R{\'e}nyi in 1960 ~\cite{ErdosRenyi}, the  random graph model has been one of the main objects of
study in probabilistic combinatorics. Given a positive integer $n$
and a real number $p \in [0,1]$, the binomial random graph $G_{n,p}$ is a random variable taking values in the set
of all labeled graphs on the vertex set $[n]=\{1, 2, \dots, n\}$. We can describe the
probability distribution of $G_{n,p}$ by saying that each pair of
elements of $[n]$ forms an edge in $G_{n,p}$ independently, with
probability $p$. Usually, the questions considered in this model
have the following generic form: given some graph property ${\cal
P}$ and a function $p = p(n) \in [0,1]$, determine if $G_{n,p}$
satisfies ${\cal P}$ {\em with high probability} (whp), i.e., if the
probability that $G_{n,p}$ satisfies ${\cal P}$ tends to $1$ as $n$
tends to infinity.

The core meta-problem in the area is the study of
the {\em evolution of $G_{n,p}$}, i.e., analyzing how $G_{n,p}$
behaves with respect to certain graph properties as $p$ traverses
the interval $[0,1]$. This task is inherently connected with the
problem of finding \emph{thresholds} guaranteeing the asymptotically almost sure
appearance (or the asymptotically almost sure non-appearance) of certain subgraphs in
$G_{n,p}$.

A graph property is said to be monotone increasing if whenever a graph $G$ satisfies it, every graph $G'\supseteq G$ satisfies it as well. For a monotone increasing property $\mathcal P_n$ (for example, ``$G_{n,p}$ contains a fixed graph $H$"), a function $q(n)$ is a threshold function for $\mathcal P_n$ if and only if

\[ \P\left[G_{n,p} \text{ satisfies } \mathcal P_n\right]\rightarrow \left\{ \begin{array}{cc}
       1&  \text{if } p(n)/q(n)=\omega(1)   \\
       0&  \text{ if } p(n)/q(n)=o(1).\end{array} \right. \]

A classical result for random graphs is that a threshold for $G_{n,p}$ to
contain a given (fixed) graph $H$ as a subgraph, is
$n^{-1/m(H)}$, where $$m(H) = \max\{|E(H')|/|V (H')| \mid H'\subseteq H, H'\neq \emptyset\}.$$
The case where $H$ is ``balanced" (that is, $m(H)=|E(H)|/|V(H)|$) had been proven in \cite{ErdosRenyi}, and for a general $H$, one can find a proof in \cite{bollobas1998random}, pages 257-274.

Alternatively, we can consider the least $p(n)$ such that for each $H'\subseteq H$, the expected number of copies of $H'$ in $G_{n,p}$ is at least one. This particular function has been named as the \emph{expectation threshold} by Kahn and Kalai in \cite{KK}.
Following \cite{KK}, note that it also makes sense if, instead of a fixed $H$,
we consider a sequence $\{H_n\}_n$, of graphs with $|V (H_n)| = n$. Formally,
for an arbitrary $H$, Kahn and Kalai defined $p_E(H)$ to be the least $p$ such that, for every spanning $H'\subseteq H$, $(|V (H')|!/|Aut(H')|)p^{|E(H')|}\geq 1$. Then, for a fixed graph $H$, the expectation threshold is the same as $p_E(H_n)$ if we consider $H_n$ as $H$ plus $n-|V(H)|$ isolated vertices.

For the class of large graphs (that is, whenever the number of vertices of $H$ is allowed to grow with $n$), one cannot expect $p_E(H)$ to always capture the true threshold behavior. For example, suppose that $H_n$ are Hamilton cycles (that is, a cycle on $n$ vertices). In this case, on one hand, the expected number of Hamilton cycles in $G_{n,p}$ is clearly $\mu_n=\frac{(n-1)!}{2}p^n$, and therefore the expectation threshold is of order $1/n$. On the other hand, it is well known \cite{posa}, that a threshold function for $H_n$ is of order $\log n/n$ (in fact, more precise results are known. The interested reader is referred to \cite{bollobas1998random} and its relevant references). A similar phenomena holds for perfect matchings (that is, a collection of $\lfloor n/2\rfloor$ pairwise disjoint edges), as was proven in \cite{ErdosRenyi}.

These examples led Kahn and Kalai to the beautiful conjecture that in fact, one cannot lose more than a $\Theta(\log |V(H)|)$ factor in $p_E(H_n)$. Specifically, they conjectured the following (Conjecture 2. in \cite{KK}).

\begin{conjecture}\label{kahn-kalai} For any $\{H_n\}_n$, a threshold function for the property ``$G_{n,p}$ contains $H_n$ " is $O(p_E(H_n)\log |V (H_n)|)$.
\end{conjecture}

Given a graph $H$ on $h$ vertices, we say that a graph $G$ on $n$ vertices contains an $H$-\emph{factor} if it contains $\lfloor n/h\rfloor$ vertex disjoint copies of $H$. For example, if $H$ is an edge, then an $H$-factor is simply a perfect matching.

Perhaps the most impressive result supporting Conjecture \ref{kahn-kalai} is due to Johansson, Kahn and Vu \cite{johansson2008factors}. In their seminal paper, Johansson, Kahn and Vu determined the threshold behavior of $H$-factors. In particular, they showed that if $H=K_{\Delta+1}$ then
\begin{align}\label{pDelta}
p(\Delta)=\left(n^{-1}\log^{1/\Delta}n\right)^{\frac{2}{\Delta+1}}
\end{align}
is a threshold function for the existence of an $H$-factor. Their results naturally extend to hypergraphs and as a corollary they proved the long standing conjecture, which is known as Shamir's Problem, of determining the threshold behavior of perfect matchings in random hypergraphs (this result garnered them the prestigious Fulkerson Prize!).
As their proof relies heavily on the facts that the connected components are small and that an $H$-factor has a lot of symmetries, it is natural to consider the following, more general, problem.

\begin{problem} Suppose that $H$ is an arbitrary graph on $n$ vertices with maximum degree $\Delta$. What is the threshold behavior of the property ``contains $H$"?
\end{problem}

Let us denote by $\mathcal H(n,\Delta)$ the family of all such graphs. A common belief is that the ``worst case threshold" for $H\in \mathcal H(n,\Delta)$ is attained when $H$ is a $K_{\Delta+1}$-factor, as this is the the most ``locally dense" graph in this family. It is thus natural to conjecture the following.

\begin{conjecture}\label{bounded}
  Let $\Delta\in \mathbb{N}$, $H\in \mathcal H(n,\Delta)$, and $p= \omega\left(n^{-1}\log^{1/\Delta}n\right)^{\frac{2}{\Delta+1}}$. Then, whp $G_{n,p}$ contains a copy of $H$.
\end{conjecture}

\begin{remark} It is worth mentioning the following:
\begin{enumerate}
\item If true, Conjecture \ref{bounded} is optimal for the family $\mathcal H(n,\Delta)$ as it is tight for $K_{\Delta+1}$-factors.
\item It may be the case that for some specific members $H\in \mathcal H(n,\Delta)$, one can do much better. For example, if $H\in \mathcal H(n,2)$ is a Hamilton cycle, as mentioned above, the threshold function for $H$ is roughly $\log n/n$, whereas from \eqref{pDelta} we obtain a bound of roughly $n^{-2/3}$. Another example is the following remarkable recent result due to Montgomery \cite{montgomery2014embedding}. He showed that for every $\Delta=O(1)$ and $H\in \mathcal H(n,\Delta)$ which is a tree, a threshold function for $H$ is of order of magnitude at most $\log^5n/n$.
\end{enumerate}
\end{remark}

In fact, we believe that something stronger is true. Let $\mathcal F$ be a family of graphs.  A graph $G$ is called $\mathcal F$-\emph{universal} if it contains every member of $\mathcal F$. We believe that the following is true.
\begin{conjecture}\label{universalconjecture}
Let $\Delta\in \mathbb{N}$ and $p=\omega \left(n^{-1}\log^{1/\Delta}n\right)^{\frac{2}{\Delta+1}}$. Then, whp $G_{n,p}$ is $\mathcal H(n,\Delta)$-universal.
\end{conjecture}

As the problem of embedding large graphs in $G_{n,p}$ is fundamental, it has been studied quite extensively,  particularly with respect to special graphs such as forests, planar graphs, and Hamilton cycles. Perhaps the first known result for general graphs from $\mathcal H(n,\Delta)$ is due to Alon and F\"uredi \cite{AF} from 1992. They introduced an embedding technique based on matchings, and showed that for a given
graph $H$ on $n$ vertices and with maximum degree $\Delta(H)=\Delta$, a
typical $G_{n,p}$ contains a copy of $H$, provided that
$p=\omega\left(n^{-1/\Delta}(\log n)^{1/\Delta}\right)$. Note that this bound on $p$ is natural, as in this range it is easy to see that whp every subset of size $\Delta$ has common neighbors in $G_{n,p}$. Therefore, this bound
has become a benchmark in the field and much subsequent work on
embedding spanning or almost-spanning structures of maximum degree
$\Delta$ in random graphs achieves this threshold.

Among all the examples we mention the following. Regarding the ``universality problem", Alon, Capalbo, Kohayakawa, R\"odl, Ruci\'nski and Szemer\'edi \cite{ACKRRS} showed that for any $\varepsilon > 0$ and any natural number $\Delta$ there exists a constant $c > 0$ such that the random graph $G_{n,p}$ is whp ~$\mathcal H((1-\varepsilon)n, \Delta)$-universal for $p \geq c (\log n/n)^{1/\Delta}$. This result was improved to $\mathcal H(n,\Delta)$ by Dellamonica, Kohayakawa, R\"odl and Ruci\'nski \cite{dellamonica2008universality} and the bound on $p$ has been slightly improved by Nenadov, Peter, and the first author \cite{FNP} in the case where all graphs are ``not locally dense". Conlon, Nenadov, Skori\'c and the first author \cite{CFNS} managed to improve the edge probability to $p=n^{-1/(\Delta-1)}\log^5n$, for the family $\mathcal H((1-\varepsilon)n,\Delta)$ with $\Delta \geq 3$ and any $\varepsilon > 0$. Recently, the first and second author along with Kronenberg, made the first step towards proving Conjecture \ref{universalconjecture} by settling the $\Delta=2$ case \cite{ferber2016optimal}.  Regarding the ``embedding one graph" problem, it is worth mentioning the very nice result of Riordan \cite{riordan2000spanning}, who managed to prove Conjecture \ref{bounded} up to a factor of $n^{\Theta(1/\Delta^2)}$ using the second moment method.

In this paper we solve Conjecture \ref{bounded} for almost-spanning graphs. That is, we prove the following result.
\begin{theorem}
  \label{thm:main}
  Let $\varepsilon, \delta>0$ be any constants and let $\Delta\geq 5$ be an integer. Then, there exists an $N_0$ such that for all $n\ge N_0$ the following holds.
For every $H\in \mathcal H((1-\varepsilon)n,\Delta)$, with probability at least $1-\delta$ the random graph $G_{n,p}$ contains a copy of $H$, provided that $p=\omega\left(n^{-1}\log^{1/\Delta}n\right)^{2/(\Delta+1)}$.
\end{theorem}

Note that for $\Delta=2$ it is relatively simple to obtain this result, and $\Delta=3$ is obtained from \cite{CFNS}. Our proof breaks down for $\Delta=4$, and it looks like a new idea is required for settling this case.
Another thing to be mentioned is that it seems like the $(\log n)^{2/(\Delta(\Delta+1))}$ factor is redundant for the almost-spanning case (for example, it is not needed if one tries to embed $(1-\varepsilon)\frac{n}{\Delta+1}$ copies of $K_{\Delta+1}$ for $\varepsilon>0$) but we have not managed to get rid of it.

{\bf Outline of the proof.} Our proof strategy is relatively simple. Given $\varepsilon >0$ and a graph $H\in \mathcal H((1-\varepsilon))n,\Delta)$, we first partition it into a ``sparse" part $H'$ and a ``dense" part (see Section \ref{sec:par}). The ``sparse" part is embedded using a result of Riordan \cite{riordan2000spanning} which is stated in Section \ref{sec:rior}. Embedding the ``dense" part is basically the heart of our proof. While partitioning $H$, we make sure that the ``dense" part consists of constantly many ``batches" of small graphs to be embedded iteratively. In an approach similar to \cite{CFNS}, in each round of this iterative embedding, we make use of a hypergraph matching theorem of Aharoni and Haxell \cite{aharoni2000hall} (see Theorem \ref{thm:hyper_match}) and Janson's inequality to show that one can extend the current embedding.

\subsection{Notation}
For a graph $G$, $V(G)$ and $E(G)$ denote its vertex set and edge set respectively. The number of vertices is denoted by $v(G) = |V(G)|$ and the number of edges $e(G) = |E(G)|$. For two vertices $x, y\in V(G)$, if $xy\in E(G)$ then we sometimes abbreviate this to $x\sim_G y$ (or $x\sim y$ if there is no risk of confusion). We denote by $\Delta(G)$ the maximum degree of the vertices in $G$.

Given two induced subgraphs $S,S'$ of a graph $G$, we let $dist_G(S,S')$ denote the \emph{distance} between them. That is, the length of the shortest path in $G$ connecting some vertex of $S$ to a vertex of $S'$ (for example, if $V(S)\cap V(S')\neq \emptyset$ then $dist_G(S,S')=0$).  Also, we let $S\cup S'$, $S \cap S'$ and $S \setminus S'$ denote the induced subgraphs in $G$ on $V(S) \cup V(S')$, $V(S) \cap V(S')$ and $V(S) \setminus V(S')$ respectively.

For two sequences $f_n$ and $g_n$ of positive numbers, we say that $f_n = O(g_n)$ or $g_n = \Omega(f_n)$ if there exists a constant $C$ such that $f_n\le C g_n$ for all $n$. We say that $f_n = o(g_n)$ or $g_n=\omega(f_n)$ if $\lim _{n\to \infty} \frac{f_n}{g_n}=0$.

\section{Auxiliary results}

In this section we present some tools and auxiliary results to be used in the proof of our main result.

\subsection{Probability estimates}

We will use lower tail estimates for random
variables which count the number of copies of certain graphs in a
random graph. The following version of
Janson's inequality, tailored for graphs, is the main estimate we will be using. This particular
statement follows immediately from Theorems $8.1.1$ and $8.1.2$ in
\cite{alon2004probabilistic}.

\begin{theorem}[Janson's inequality] \label{thm:Janson}
Let $p \in (0, 1)$ and let $\{ H_i \}_{i \in
\mathcal{I}}$ be a family of subgraphs of the complete graph on the vertex set
$[n]$. For each $i \in \mathcal{I}$, let $X_i$
denote the indicator random variable for the event that $H_i \subseteq
G_{n,p}$ and, for each ordered pair $(i, j) \in \mathcal{I} \times
\mathcal{I}$ with $i \neq j$, write $H_i \sim H_j$ if $E(H_i)\cap
E(H_j)\neq \emptyset$. Then, for
\begin{align*}
X &= \sum_{i \in \mathcal{I}} X_i,\\
\mu &= \mathbb{E}[X] = \sum_{i \in \mathcal{I}} p^{e(H_i)},\\
\delta &= \sum_{\substack{(i, j) \in \mathcal{I} \times \mathcal{I} \\ H_i \sim H_j}} \mathbb{E}[X_i X_j] = \sum_{\substack{(i, j) \in \mathcal{I} \times \mathcal{I} \\ H_i \sim H_j}} p^{e(H_i) + e(H_j) - e(H_i \cap H_j)}
\end{align*}
and any $0 < \gamma < 1$, we have
$$ \P[X < (1 - \gamma)\mu] \le e^{- \frac{\gamma^2 \mu^2}{2(\mu + \delta)}}. $$
\end{theorem}
\subsection{Partitioning $H$}\label{sec:par}

In order to embed a graph $H\in \mathcal H((1-\varepsilon))n,\Delta)$ in $G_{n,p}$, we first wish to partition it ``nicely" in a way which will be convenient for us to describe the embedding. Before stating it formally, we need the following definition.

\begin{definition}\label{def:good}
  Let $H\in \mathcal H(m,\Delta)$, let $\mathcal S_1,\ldots, \mathcal S_k$ be a collection of families of induced subgraphs of $H$ and let $H':=H\setminus \left(\cup_i\cup_{S\in \mathcal S_i} S\right)$. We say that the partition $(H',\mathcal S_1,\ldots,\mathcal S_k)$ is $\alpha$-\emph{good} for some $\alpha>0$ if and only if the following hold.
  \begin{enumerate}
    \item For every $1\le i\le k$, all the graphs in $\mathcal S_i$ are isomorphic in $H$. More precisely, for any two graphs $S$ and $S'$ in $\mathcal S_i$, there exists a way to label the vertices $V(S) = \{ v_1, \dots, v_g\}$ and $V(S') = \{v_1', \dots, v_g'\}$ such that $v_j\sim v_l$ if and only if $v_j'\sim v_l'$ and $\deg_H(v_j) = \deg_H(v_j')$ for any $j$;
    \item Each $\mathcal S_i$ contains at most $\alpha n$ graphs;
    \item For every $1\leq i\leq k$ and every graph $S\in \mathcal S_i$, $v(S)\ge 3$ and
    $$\frac{e(S)}{v(S)-2}>\frac{\Delta+1}{2}\geq \frac{e(S')}{v(S')-2}$$
    for all proper subgraphs $S'\subset S$ with $v(S')\geq 3$;
    \item Every two distinct graphs in $\cup_i \mathcal S_i$ are vertex disjoint;
    \item For every $i$ and every $S,S'\in \mathcal S_i$ we have $dist_H(S,S')\geq 3$; \label{distcond}
    \item The graph $H'$ has ``small" density, namely, $\frac{e(S)}{v(S)-2}\leq \frac{\Delta+1}{2}$ for every subgraph $S\subseteq H'$ with $v(S)\geq 3$.
  \end{enumerate}
\end{definition}

\begin{lemma}
  \label{lemma:Partitioning}
Let $\alpha\in (0, 1)$, $\Delta\in \mathbb{N}$ and let $H\in \mathcal H(m,\Delta)$ with $m$ being sufficiently large. Then, there exists an $\alpha$-good partition $(H',\mathcal S_1,\ldots,\mathcal S_k)$ of $H$ with $k\leq k(\alpha, \Delta)$ where $k(\alpha, \Delta)$ is some constant depending only on $\alpha$ and $\Delta$.
\end{lemma}

\begin{proof}
First, we find the graph $H'$ greedily as follows.
  Let $\mathcal S:=\emptyset$ and $H_0:=H$. As long as there exists an induced subgraph $S\subseteq H_i$ with at least three vertices and $e(S)/(v(S)-2)>(\Delta+1)/2$, pick a minimal (with respect to inclusion) such graph $S$ and update $\mathcal S:=\mathcal S\cup \{S\}$ and let $H_{i+1}:=H_i\setminus S$. Note that this procedure must terminate at some step $\ell$, and let $H':=H_{\ell}$. Observe that $e(S)/(v(S)-2)\leq (\Delta+1)/2$ for every $S\subseteq H'$ with $v(S)\ge 3$ (this verifies (vi) in the definition of $\alpha$-goodness).

  Second, note that since $\Delta(H)\leq \Delta$, it follows that for every subgraph $S\subseteq H$ we have $e(S)\le v(S)\Delta/2$. In particular, for every $S\in \mathcal S$, together with the inequality $e(S)/(v(S)-2)>(\Delta+1)/2$, we conclude that $v(S)\leq 2\Delta +1$.

  Third, define an auxiliary graph $\mathcal A$ with vertex set $\mathcal S$, where for $S$ and $S'$ in $\mathcal S$, $SS'$ is an edge in $\mathcal A$ if and only if $dist_H(S,S')\leq 2$. Note that since $\Delta(H)\leq \Delta$ and since $v(S)\leq 2\Delta+1$ for every $S\in \mathcal S$, it follows that the maximum degree of $\mathcal A$ is $\Delta(\mathcal A)\leq (2\Delta+1)\Delta^{2}\le 3\Delta^3$. Therefore, one can (greedily) find a proper coloring of the vertices of $\mathcal A$ using $c=3\Delta^3+1$ colors. Let $\mathcal T_i$ denote the color class $i$. Observe that for every $S,S'\in \mathcal T_i$, $dist_H(S,S')\geq 3$.

  Next, since $v(S)\leq 2\Delta+1$ for all $S\in\mathcal S$, it follows that there are at most (say) $s:=\Delta^{2\Delta+1}2^{{2\Delta+1 \choose 2}}$ possible distinct such graphs (up to graph isomorphism in $H$), and therefore, one can further divide each of the classes $\mathcal T_i$ into at most $s$ subsets, each of which consists of isomorphic graphs.

  Finally, if any of these subsets contain more than $\alpha n$ graphs, we further partition each of them into at most $\alpha^{-1}$ subsets. All in all, by relabeling, for $k\leq sc\alpha^{-1}$, one obtains an $\alpha$-good partition $\mathcal S_1,\ldots,\mathcal S_k$ and $H'$ as desired.
\end{proof}

\subsection{Embedding the ``sparse" part}
\label{sec:rior}
Let $H\in \mathcal H(m,\Delta)$ and let $(H',\mathcal S_1,\ldots,\mathcal S_k)$ be an $\alpha$-good partition of $H$ as obtained by Lemma \ref{lemma:Partitioning}. Note that, by definition of $\alpha$-goodness,  for every $S\subseteq H'$ with at least $3$ vertices, one has
$$e(S)/(v(S)-2)\leq (\Delta+1)/2.$$

 In order to embed $H'$, we make use of the following theorem due to Riordan (an immediate corollary of \cite[Theorem 5]{riordan2000spanning}).  We use $G_{n,M}$ to denote the random variable uniformly distributed over all labeled graphs on the vertex set $[n]$ with exactly $M$ edges.
\begin{theorem}\label{riordan}
Let $\delta>0$ and let $\Delta\geq 2$ be a positive integer. Then, there exists an $N_0$ such that for all $n\ge N_0$ the following holds.
For every $H\in \mathcal H(n,\Delta)$, let
$$\gamma(H) = \max_{S\subseteq H, v(S)\ge 3}\frac{e(S)}{v(S)-2}.$$
Then with probability at least $1-\delta$ the random graph $G_{n,M}$ contains a copy of $H$, provided that $M = p'{n \choose 2}$ with $p'\geq n^{-1/\gamma}\log\log n$.
\end{theorem}

Note that the property ``containing $H$" is a monotone increasing property, and since whp $e(G_{n,p})\geq (1-o(1))\binom{n}{2}p$ if $\binom{n}{2}p = \omega(1)$, it follows that the same conclusion of Theorem \ref{riordan} holds for $G_{n,p}$, provided that (say) $p\geq 2n^{-1/\gamma}\log\log n$.

\subsection{Embedding the small dense graphs}\label{sec:emb}

The following technical lemma is the primary ingredient in the proof of our main result. Roughly speaking, given an $\alpha$-good partition of $H$, we will embed the graphs from the sets $\mathcal S_i$ with a hypergraph matching theorem (see Theorem \ref{thm:hyper_match}). Before making the exact statement, we need the following preparation.

Let $\varepsilon>0$ and $\Delta\geq 5$. Let $H\in \mathcal H((1-\varepsilon)n,\Delta)$, and let $(H',\mathcal S_1,\ldots,\mathcal S_k)$ be an $\alpha$-good partition of $H$ where $\alpha=\varepsilon/\left (3(2\Delta+1)^2\right )$ and $k\le k(\alpha,\Delta)$ (the existence of such a partition is ensured by Lemma \ref{lemma:Partitioning}). Our strategy in the proof of Theorem \ref{thm:main} is to expose $G_{n,p}$ in $k+1$ rounds $0\leq h\leq k$ as follows. In round $h=0$, we embed $H'$ into a $G_0:=G_{n,q}$, where $q$ is the unique number $q\in (0,1)$ satisfying $(1-q)^{k+1}=1-p$. For each $0\leq h\leq k-1$, in round $h+1$, we extend our current embedding $f$ of $H_h:=H'\cup \left(\cup_{S\in \mathcal S_i,i\leq h}S\right)$ into an embedding of $H_{h+1}:=H_h\cup \left(\cup_{S\in \mathcal S_{h+1}}S\right)$ using the edges of a graph $G_{h+1}=G_{n,q}$ which is independent of all the other $G_j$.  Here we abuse notation and let $f$ denote the partial embedding at each stage.  Our goal in this section is to prove the following key lemma which enables us to extend $f$.

\begin{lemma}
  \label{key lemma}
With the above setting, for every $h\leq k-1$, assuming that $f$ is an embedding of $H_h$ into $\cup_{i\leq h}G_i$, whp the embedding $f$ can be extended into an embedding of $H_{h+1}$ using the edges of $G_{h+1}$.
\end{lemma}

Clearly, given an embedding of $H'$ into $G_0$, by iterating Lemma \ref{key lemma} $k$ times, one can (whp) extend $f$ into an embedding of $H$ into $\cup_{h=0}^{k}G_h$ which has the exact same distribution as a $G_{n,p}$.

To prove Lemma \ref{key lemma}, we make use of the following Hall-type argument for finding large matchings in hypergraphs due to Aharoni and Haxel \cite{aharoni2000hall}. This idea is based on the one in \cite{CFNS}, although here we have to be more careful in order to obtain the correct bound on $p$. We now state the theorem and describe how to use it in our setting.

\begin{theorem}[Hall's criterion for hypergraphs, Corollary 1.2 \cite{aharoni2000hall}]  \label{thm:hyper_match}
Let $\{L_1, \dots, L_t\}$ be a family of $s$-uniform hypergraphs on the same vertex set.
If, for every $\mathcal{I} \subseteq [t]$, the hypergraph $\bigcup_{i \in \mathcal{I}} L_i$ contains a matching of size greater than $s(|\mathcal I| - 1)$, then there exists a function $g:[t] \rightarrow \bigcup_{i = 1}^t E(L_i)$ such that $g(i) \in E(L_i)$ and $g(i) \cap g(j) = \emptyset$ for $i \neq j$.
\end{theorem}

To apply Theorem \ref{thm:hyper_match} to our setting in Lemma \ref{key lemma}, we proceed as follows.
First, let us enumerate $\mathcal S_{h+1}=\{S_1,\ldots,S_t\}$ and recall that since our partition is $\alpha$-good, all the graphs $S_i$ in $\mathcal S_{h+1}$ are isomorphic (henceforth, we refer to all the graphs $S_i$ as $S$). Moreover, letting $s=v(S)$, from the definition of $\alpha$-goodness and the proof of Lemma \ref{lemma:Partitioning}, we have
\begin{equation}
  t\le \alpha n,\label{boundt}
  \end{equation}
  \begin{equation}
  s\le 2\Delta + 1,\label{boundg}
  \end{equation}
  and for any proper subgraph $S'$ of $S$ with $v(S') \geq 3$,
\begin{equation}
 \frac{e(S')}{v(S')-2}\le \frac{\Delta+1}{2}<\frac{e(S)}{s-2}.\label{removalcondition}
 \end{equation}

Second, suppose that $V(S)=[s]$, and for every $i\leq t$, let us enumerate the vertex set of $S_i$ as $v_{i1}, \dots, v_{is}$ in such a way that the map $\varphi(j)=v_{ij}$ is a graph isomorphism between $S$ and $S_i$, and $\deg_{H}(v_{ij})=\deg_{H}(v_{i'j})$, for all $i, i'$ and $j$. Let $D_h = [n]\setminus f(V(H_{h}))$. Since $v(H_h)\le (1-\varepsilon)n$, it follows that $|D_h|\ge \varepsilon n$. Recall that
\begin{equation}
\alpha = \frac{\varepsilon}{3(2\Delta+1)^{2}}.\label{def_alpha}
\end{equation}
Clearly, we have that
\begin{equation} \label{obs:1}
|D_h| - s^{2}t\ge \varepsilon n - s^{2}\alpha n\ge \varepsilon n/2,
 \end{equation}
and this will be useful in proving Lemma \ref{lemma:concentration} below.

Third, let $W_{ij}$ be the (image of the) set of neighbors of $v_{ij}$ in $H_{h}$. Note that all the $W_{ij}$ are fixed because $H_{h}$ has already been embedded and there are no edges between graphs in $\mathcal S_{h+1}$. Moreover,  by (v) of Definition \ref{def:good}, the family $\{W_{i j} \}_{ij \in [t] \times [s]}$ satisfies
\begin{enumerate}[(i)]
\item $W_{ij} \subseteq [n] \setminus D_h$, and $|W_{i j}|\le d_j:=\Delta - \indeg(v_{ij})$ where $\indeg(v_{ij})$ is the degree of $v_{ij}$ in $S_i$, and
\item $W_{i j} \cap W_{i' j'} = \emptyset$ for all $i \neq i'$.
\end{enumerate}

Roughly speaking, the larger the sets $W_{ij}$ are, the harder it is to embed $\mathcal S_{h+1}$. Thus, without loss of generality, we can assume that $W_{ij}$ contains all the neighbors of $v_{ij}$ in $H$, and so, by (i) of Definition \ref{def:good}, $|W_{ij}| = |W_{i'j}|$ for all $i, i'$, and $j$. One can reduce the general case to this case by temporarily moving vertices from $D_{h}$ to $W_{ij}$ so that $|W_{ij}| = \deg_H(v_{ij}) - \indeg(v_{ij})$ for all $i,j$. Since there are $st$ vertices in $\cup_{S\in \mathcal S_{h+1}} V(S)$, the number of moved vertices is at most $\Delta st$ and the number of remaining vertices in $D_{h}$ is still greater than $\eps n/2 + s^{2}t$. That means \eqref{obs:1} still holds, which is enough for us to implement the embedding of $H_{h+1}$.

Next, for each $1\le i\le t$, let $L_i$ be the $s$-uniform hypergraph on the vertex set $D_{h}$ where $(v_1, \dots, v_s)\in D_h^{s}$ is a hyperedge if and only if, in $G_{h+1}=G_{n,q}$ (the new random graph with edge probabillity $q$), $v_j$ is connected to every vertex in $W_{ij}$ and $v_{ij}v_{ij'}\in E(S_i)$ implies $v_jv_{j'}\in E(G_{h+1})$. That is, every hyperedge $e$ of $L_i$ corresponds to an embedding of $S_i$ which extends $f$. A moment's thought now reveals that extending $f$ into an embedding of $H_{h+1}$ is equivalent to showing that there exists a function $g:[t] \rightarrow \bigcup_{i = 1}^t E(L_i)$ such that $g(i) \in E(L_i)$ and $g(i) \cap g(j) = \emptyset$ for every $i \neq j$. To this end, in order to check that the assumptions of Theorem \ref{thm:hyper_match} are satisfied (and therefore, such a $g$ exists), we need the following technical lemma.

\begin{lemma}\label{lemma:concentration} For every $1\leq h\leq k$ and $1\le r\le t$, with probability at least $1 - \exp\left (-\omega(r\log n)\right )$, the following holds. For every $D'\subseteq D_h$
of size $|D'|\ge |D_h| - s^{2}r$, and every $\mathcal I\subseteq [t]$ of size $|\mathcal I|=r$, $G_{h+1}$  contains a copy $S'$ of $S$ with $V(S') = \{v_1, \ldots, v_s\}$ in which $v_j$ is the copy of $j$ for each $j \in [s] = V(S)$.  Furthermore, $v_j\in D'$ for all $1\le j\le s$ and for some $i \in \mathcal I$, $W_{i j}$ is contained in the neighborhood of $v_j$ for all $1\le j\le s$.
\end{lemma}

In the rest of this section, we prove Lemma \ref{lemma:concentration} and deduce from it Lemma \ref{key lemma}.
\begin{proof}[of Lemma \ref{lemma:concentration}]
Let us fix $D'$ and $\mathcal I$ as in the assumptions of the lemma. Note that $s^2r\leq s^2 t\leq \varepsilon n/2$ and $|D_h|\leq n$. Therefore, there are at most $s^{2}r\binom{n}{n-s^2r}\binom{n}{r}\leq n\exp\left(s^2r\log n+r\log n\right)=\exp\left(\Theta(r\log n)\right)$ ways to choose $D'$ and $\mathcal I$. Since the tail probability we aim to show is $\exp\left(-\omega(r\log n)\right)$, this allows us to take a union bound over $D'$ and $\mathcal I$. Without loss of generality, we can assume $\mathcal I = [r]$.

Next, for $i\in [r]$ and for a copy $S'$ of $S$ with $V(S') = \{v_1, \dots, v_s\}$, we define the graph $S' \oplus i$ with vertex set and edge set
\begin{align*}
V( S' \oplus i) = V(S') \cup \bigcup_{j = 1}^s W_{ij}, \quad \text{and} \quad
E(S' \oplus i) = E(S') \cup \{v_jv \mid j \in [s] \text{ and }v \in W_{ij}\}.
\end{align*}
Furthermore, decreasing the size of $D'$ by at most $s-1$, we can assume that $s$ divides $|D'|$.  Let $D' = V_1\cup \ldots\cup V_s$ be a partition of $D'$ into equal sized sets and define the family of \emph{canonical graphs} $\mathcal{S}$ as
$$ \mathcal{S} := \{ S' = (v_1, \ldots, v_s) :  S' \mbox{ is a copy of } S \mbox{ and } v_j \in V_j \; \text{for all} \; j \in [s] \}$$
and set
$$ \mathcal{S}^+ := \{ S' \oplus i : S' \in \mathcal{S} \; \text{and} \;  i \in [r]\}. $$
Observe that if $G$ contains any graph from $\mathcal{S}^+$, then $G$ contains the desired copy. In order to bound the probability for $G$ not to contain such a graph, we apply Janson's inequality (Theorem \ref{thm:Janson}). Thus, we need to estimate the parameters $\mu$ and $\delta$ appearing in Theorem \ref{thm:Janson}.

For every $S'=(v_1,\ldots,v_s)\in \mathcal S$, we have
\begin{equation}
\Delta s = \sum_{j=1}^{s}(d_j + \indeg(v_j)) =  \sum_{j=1}^{s} d_j + 2e(S).\label{sumdeg}
\end{equation}

Note that each graph $S^+ \in \mathcal{S}^+$ appears in $G_h$ with probability $q^{e(S^{+})}$, where $e(S^+) = e(S) + \sum_{j=1}^{s}|W_{i j}|\le  e(S) + \sum_{j=1}^{s} d_j$. Clearly,
$$\mu = |\mathcal{S}^+|  q^{e(S^+) } = r \left(\frac{|D'|}{s} \right)^s  q^{e(S^+)}.$$

Our goal is to show that $\mu =\omega(r\log n)$ and $\delta = o(\mu^2r^{-1} \log^{-1}n)$. Indeed, assuming so, by Janson's inequality we obtain
\begin{equation}
\P(G \text{ does not contain any } S^+\in \mathcal S^+)\le \exp\left (-\frac{\mu^2}{4(\mu+\delta)}\right )= \exp(-\omega(r\log n))\nonumber
\end{equation}
as desired.

First, we show that $\mu = \omega(r\log n)$. By our choice of $\alpha$ in \eqref{def_alpha}, it follows that $|D'|\ge \varepsilon n/2$. Moreover, as $p=o(1)$, $k$ is a constant, and $(1-q)^{k+1}=1-p$, it follows that $q=\Theta(p)$. Therefore, we have that
$$\mu = \Omega(rn^{s} q^{e(S) + \sum_{j=1}^{s} d_j}) = \omega\left (rn^{s - \frac{2e(S) + 2\sum_{j=1}^s d_j}{\Delta+1}} (\log n)^{2(e(S) +\sum_{j=1}^s d_j)/(\Delta(\Delta+1))}\right ).$$
In order to estimate $\mu$, we show that
\begin{equation}
\Delta(\Delta+1)\le 2e(S) + 2\sum_{j=1}^s d_j\le (\Delta+1)s,\label{bounde}
\end{equation}
from which we conclude that
$$\mu = \omega\left (r\log n\right ),$$ as desired.
By equation \eqref{sumdeg}, proving \eqref{bounde} is equivalent to showing that
\begin{equation}
\Delta(\Delta+1-s)\le \sum_{j=1}^{s} d_j\le s. \label{cond}
\end{equation}
Recall that by property \eqref{removalcondition} we have $$\frac{e(S)}{s-2}>\frac{\Delta+1}{2}.$$
Therefore, combining this with \eqref{sumdeg}, we obtain that $$2e(S)>(\Delta+1)s - 2(\Delta+1) = 2 e(S) +\sum_{j=1}^s d_j + s - 2(\Delta+1),$$
which yields
\begin{equation}
s + \sum_{j=1}^s d_j \le 2\Delta+1.\label{sumg}
\end{equation}

If $s\ge \Delta+1$, then clearly $0 \leq \sum_{j=1}^s d_j \le \Delta < s$, proving \eqref{cond}.\\

If $s = \Delta$, then by \eqref{sumg}, $\sum_{j=1}^s d_j\le \Delta + 1$. From \eqref{sumdeg}, $\sum_{j=1}^s d_j$ has the same parity as $\Delta s = \Delta^{2}$ and therefore it cannot be $\Delta + 1$. Moreover, since $e(S)\le {s\choose 2}$, from \eqref{sumdeg}, we also find that $\sum_{j=1}^s d_j\ge \Delta$. Thus, $\sum_{j=1}^s d_j = \Delta$, proving \eqref{cond}.\\

Finally, we show that $s> \Delta-1$. If $s = 3$, due to the assumptions $\frac{e(S)}{s-2}>\frac{\Delta+1}{2}$ and $\Delta\ge 5$, $e(S)$ has to be at least $4$ which is impossible. Hence, $s\ge 4$. This bound, along with the condition that $\frac{e(S)}{s-2} > \frac{\Delta +1}{2}$, implies that
\begin{eqnarray}
(s-2)(s+2) = s^2 -4 \geq s^2 - s \geq 2 e(S) > (\Delta+1)(s-2)\label{sD}
\end{eqnarray}
This gives that $s+2 > \Delta + 1$, so $s > \Delta - 1$, completing the proof of \eqref{cond} and hence \eqref{bounde}.\\

To prove that $\delta = o(\mu^2 r^{-1}\log^{-1}n)$, we have
$$ \delta = \sum_{i, j \in [r]} \sum_{\substack{S', S'' \in \mathcal{S}\\S' \oplus i \sim S'' \oplus j}} q^{e(S' \oplus i) + e(S'' \oplus j) - e((S' \oplus i) \cap (S'' \oplus j))}. $$
We consider the two cases $i = j$ and $i \neq j$.

First, let us consider the case $S' \oplus i \sim S'' \oplus i$, for some $i \in [r]$ and distinct graphs $S', S'' \in \mathcal{S}$. Let $J := S' \cap S''$ and observe that $1\le v(J) \le s - 1$. Let $\mathcal{J}_1$ be the family consisting of all possible graphs of the form $S' \cap S''$. That is,
$$ \mathcal{J}_1 := \{J = S' \cap S'' : S', S'' \in \mathcal{S} \; \text{and} \; v(J) \in \{1, \ldots, s - 1\}\}. $$

For each $j\in J$, there is a unique $\varphi(j)\in [s]$ such that $j\in V_{\varphi(j)}$. For the rest of this proof of Lemma \ref{lemma:concentration}, we abuse notation and write $d_j$ for $d_{\varphi(j)}$, $W_{i,j}$ for $W_{i, \varphi(j)}$ and $J$ for $J':= \{\varphi(j): j \in J \}$.

We have
$$
	e(S' \oplus i \cap S'' \oplus i) = e(J) + \sum_{j\in J} |W_{i, j}|\le e(J) + \sum_{j\in J} d_j.
$$
With these observations in hand, we can bound the contribution of such pairs to $\delta$ as follows:

\begin{eqnarray}
\delta_1  := \sum_{i \in [r]} \sum_{\substack{S', S'' \in \mathcal{S}\\v(S' \cap S'') \ge 1}}  q^{e(S' \oplus i) + e(S'' \oplus i) - e(S' \oplus i \cap S'' \oplus i)}
 \le \sum_{i \in [r]} \sum_{J \in \mathcal{J}_1} \sum_{\substack{S', S'' \in \mathcal{S}\\S' \cap S'' = J}} q^{2 e({S^+}) - ( e(J) + \sum_{j\in J} d_j)}\nonumber
 \end{eqnarray}
 and so
 \begin{align}
 \begin{split}
\delta_1 &\le  r \sum_{J \in \mathcal{J}_1} \left(\frac{|D'|}{s} \right)^{2(s - v(J))}  q^{2e({S^+}) - ( e(J) + \sum_{j\in J} d_j)} = \frac{\mu^2}{r} \sum_{J \in \mathcal{J}_1}  \left(\frac{|D'|}{s} \right)^{-2 v(J)}  q^{-(e(J) + \sum_{j\in J} d_j)}.\nonumber
\end{split}
\end{align}

Therefore, $$\delta_1\le \delta_{1,1} + \delta_{1,2}$$
where
\begin{eqnarray}
\delta_{1, 1} := \frac{\mu^2}{r} \sum_{J \in \mathcal{J}_1, v(J) \leq s-3}  \left(\frac{|D'|}{s} \right)^{-2 v(J)}  q^{-(e(J) + \sum_{j\in J} d_j)} \label{delta11}
 \end{eqnarray}
and
\begin{eqnarray}
\delta_{1,2}:=\frac{\mu^2}{r} \sum_{J \in \mathcal{J}_1, v(J) \ge s- 2}  \left(\frac{|D'|}{s} \right)^{-2 v(J)}  q^{-(e(J) +  \sum_{j\in J} d_j)}. \label{delta12}
\end{eqnarray}

Let us now show that the power $e(J) + \sum_{j\in J} d_j$ cannot be large in $\delta_{1,1}$. Since $I: = S\setminus J$ is a proper subgraph of $S$, and $v(I) \geq 3$, we have from assumption \eqref{removalcondition} that
\begin{equation}
 \frac{e(I)}{v(I)-2}\le \frac{\Delta+1}{2}<\frac{e(S)}{s-2}.\nonumber
 \end{equation}

 Rearranging, we get
 \begin{eqnarray}
 (\Delta +1)v(I) - 2 e(I) \ge 2(\Delta + 1) >  (\Delta +1)s - 2 e(S) ,\nonumber
 \end{eqnarray}
which gives
\begin{equation}
\Delta v(J)< 2 e(J, I) + 2 e(J)-v(J)\nonumber
\end{equation}
where $e(J, I)$ is the number of edges between $I$ and $J$.

Combining that with
$$\Delta v(J) = 2 e(J) + \sum_{j\in J} d_j +e(J, I),$$
yields
\begin{eqnarray}
\sum_{j\in J} d_j \le  e(J, I)-v(J) - 1\nonumber
\end{eqnarray}
and so
\begin{eqnarray}
e(J) + \sum_{j\in J} d_j \le \frac{2 e(J)+\sum_{j\in J} d_j  +e(J, I)-v(J) - 1}{2} = \frac{(\Delta-1)v(J)-1}{2}< \frac{(\Delta-1)v(J)}{2}.\nonumber
\end{eqnarray}

Plugging this bound into \eqref{delta11}, we obtain
\begin{eqnarray}
\delta_{1,1}   & \le& \frac{\mu^2}{r} \sum_{J \in \mathcal{J}_1, v(J) \leq s- 3}  \left(\frac{|D'|}{s} \right)^{-2 v(J)}  q^{-\frac{\Delta-1}{2} v(J)}.\nonumber
\end{eqnarray}
Note that there are at most $\binom{s}{u} \left( |D'| / s \right)^{u}$ graphs $J \in \mathcal{J}_1$ with $v(J) = u$. Thus, we have
\begin{eqnarray}
\delta_{1,1} \le \frac{\mu^2}{r} \sum_{u = 1}^{s-3} \binom{s}{u} \left( \frac{|D'|}{s} \right)^{-u} q^{-\frac{(\Delta-1)u}{2}} = O(\mu^2 r^{-1} n^{-\frac{2}{\Delta+1}}) = o(\mu^2 r^{-1}\log^{-1}n).\nonumber
\end{eqnarray}

For $\delta_{1,2}$, since $s>v(J)\ge s-2$, there are only two cases, $v(J) = s-2$ and $v(J)=s-1$. Let $I = S\setminus J$ as above. We have $v(I)\le 2$ and so $e(I)\le 1$. Therefore,
$$e(S) = e(J) + e(I) + e(J, I)\le e(J) + e(J, I)+1.$$
Plugging this into the inequality $\frac{e(S)}{s-2}>\frac{\Delta+1}{2}$, we get
$$(\Delta+1)(s-2)\le 2e(S)-1 \le 2e(J)+2e(J, I)+1.$$
Thus,
\begin{eqnarray}
2\left (e(J) + \sum_{j\in J} d_j\right ) = 2\Delta v(J) - 2e(J) - 2 e(J, I) \le 2\Delta v(J) - (\Delta+1)(s-2)+1.\nonumber
\end{eqnarray}
Since we proved in \eqref{sD} that $s\ge \Delta\ge 5$, we have
$$(\Delta-1/2)v(J)\le (\Delta-1/2)(s-1)\le (\Delta+1)(s-2).$$
And so,
$$2\left (e(J) + \sum_{j\in J} d_j\right )\le (\Delta+1/2) v(J)+1.
$$
Plugging this into \eqref{delta12} and estimating the number of graphs $J \in \mathcal{J}_1$ with $v(J) = u$, we get
\begin{eqnarray}
\delta_{1,2} \le \frac{\mu^2}{r} \sum_{u = s-2}^{s-1} \binom{s}{u} \left( \frac{|D'|}{s} \right)^{-u} q^{-\frac{(\Delta+1/2)u}{2}-\frac{1}{2}} = O(\mu^2 r^{-1} n^{-\frac{1/2}{\Delta+1}}) = o(\mu^2 r^{-1}\log^{-1}n).\nonumber
\end{eqnarray}
Therefore, we have
$$
\delta_1 \leq \delta_{1,1} + \delta_{1,2} = o(\mu^2 r^{-1} \log^{-1} n).
$$

Next, if $S' \oplus i \sim S'' \oplus j$ for $i \neq j$ and $S',
S'' \in \mathcal{S}$, then we have $(S' \oplus
i) \cap (S'' \oplus j) = S' \cap S''$. Let $J := S' \cap S''$ and
observe that $e(J) \ge 1$, as otherwise $S' \oplus i$ and $S'' \oplus
j$ would not have any edges in common. Let $\mathcal{J}_2$ be the
family consisting of all possible graphs of the form $S' \cap S''$,
$$ \mathcal{J}_2 := \{J = S' \cap S'' : S', S'' \in \mathcal{S} \; \text{and} \; e(J) \ge 1\}. $$
The contribution of such pairs to $\delta$ is
\begin{equation}
\delta_2 :=  \sum_{i \neq j} \sum_{\substack{S', S'' \in \mathcal{S}\\e(S' \cap S'') \ge 1}}  q^{e(S' \oplus i) + e(S'' \oplus j) - e(S' \cap S'')} =  \sum_{i \neq j} \sum_{J \in \mathcal{J}_2} \sum_{\substack{S', S'' \in \mathcal{S}\\S' \cap S'' = J}}  q^{2 e({S^+})  - e(J)} \nonumber
\end{equation}
and can be bounded by
\begin{equation}
\delta_2 \le r^{2} \sum_{J \in \mathcal{J}_2, 2\le v(J)\le s} \left (\frac{|D'|}{s}\right )^{2(s - v(J))} q^{2 e({S^+})  - e(J)}=\mu^2  \sum_{J \in \mathcal{J}_2, 2\le v(J)\le s} \left (\frac{|D'|}{s}\right )^{-2v(J)} q^{- e(J)}\nonumber
\end{equation}
which can be further split into two terms when $S'\neq S''$ and when $S' = S''$,
\begin{equation}
\delta_2\le \mu^2  \sum_{J \in \mathcal{J}_2, 2\le v(J)<s} \left (\frac{|D'|}{s}\right )^{-2v(J)} q^{- e(J)}  + \mu^2 \sum_{S\in \mathcal S}   \left (\frac{|D'|}{s}\right )^{-2s}q^{ - e(S)} =: \delta_{2,1} + \delta_{2,2}\nonumber.
\end{equation}

For $3\le v(J)<s$ and $J\subset S$ with $S\in \mathcal S$, by property \eqref{removalcondition}, we have
\begin{equation}
e(J) \le \frac{\Delta+1}{2}( v(J)-2) .\nonumber
\end{equation}
When $v(J)=2$, $e(J) = 1$ by the definition of $\mathcal{J}_2$.
Thus,
\begin{eqnarray}
\delta_{2, 1}&=O(\mu ^{2})n^{-2}q^{-1}+O(\mu ^{2})\sum_{u = 3}^{s-1} n^{-u}q^{-\frac{\Delta+1}{2}( u-2)} = o\left (\mu^2 n^{-\frac{2\Delta}{\Delta+1}}\right )=o\left (\mu^2 r^{-1}\log^{-1}n\right ).\nonumber
\end{eqnarray}
For $\delta_{2, 2}$, we have
$$\delta_{2, 2}=O\left ( \mu^2 n^{-s} q^{-e(S)}\right )=o\left (\mu^2 n^{-s + \frac{2e(S)}{\Delta+1}}(\log n)^{-\frac{2e(S)}{\Delta(\Delta+1)}}\right ).$$
Recall that $s\ge \Delta$.
If $s\ge \Delta+1$ then
$$n^{-s + \frac{2e(S)}{\Delta+1}}\le n^{-s+\frac{\Delta s}{\Delta+1}} = n^{-\frac{s}{\Delta+1}}\le n^{-1}\le r^{-1}.$$
In these estimates, unless $s = \Delta+1$ and $e(S) = s\Delta/2$, we have
$$n^{-s + \frac{2e(S)}{\Delta+1}}\le n^{-1-\frac{1}{\Delta+1}}\le r^{-1}\log^{-1}n.$$
and hence
$$\delta_{2, 2} = o\left (\mu^2 r^{-1}\log^{-1}n\right ).$$
In the case that  $s = \Delta+1$ and $e(S) = s\Delta/2$, that is $S$ is a clique on $\Delta+1$ vertices, we still have
\begin{equation}
\delta_{2, 2}= o\left ( \mu^2 n^{-s + \frac{2e(S)}{\Delta+1}}(\log n)^{-\frac{2e(S)}{\Delta(\Delta+1)}}\right ) = o\left (\mu^2 r^{-1}\log^{-1}n\right )\nonumber.
\end{equation}
If $s = \Delta$ then
$$e(S)\le {s\choose 2} = (\Delta-1)\Delta/2.$$
Thus,
$$n^{-s + \frac{2e(S)}{\Delta+1}}\le n^{-\frac{2\Delta}{\Delta+1}}\le n^{-1-\frac{\Delta-1}{\Delta+1}}\le r^{-1}n^{-\frac{\Delta-1}{\Delta+1}}.$$
In all cases, we obtain
$$\delta_2 = o(\mu^2 r^{-1} \log^{-1}n).$$
To sum up, we have
$$\delta \le \delta_1 + \delta_2 = o(\mu^2 r^{-1}\log^{-1}n),$$
completing the proof of the lemma.
\end{proof}

Now we are ready to prove Lemma \ref{key lemma}.
\begin{proof}[of Lemma \ref{key lemma}] All that is left to do is to check that the condition in Theorem \ref{thm:hyper_match} is satisfied whp. By the union bound over $1\le t\le \alpha n$ and $1\le r\le t$, the conclusion of Lemma \ref{lemma:concentration} holds for all $r$ and $t$ simultaneously except on a set of probability at most
$$\sum_{t=1}^{n} \sum _{r=1}^{t} e^{-\omega (r\log n)}\le n^{2-\omega(1)} = o(1).$$
Assume that the conclusion of Lemma \ref{lemma:concentration} holds for all $1\le t\le \alpha n$ and $1\le r\le t$. We want to show that for every $\mathcal{I} \subseteq [t]$, the hypergraph $\bigcup_{i \in \mathcal{I}} L_i$ contains a matching of size greater than $s(|\mathcal I| - 1)$. Let $r= |\mathcal I|$. Let $D' = D_h$. By the property from Lemma \ref{lemma:concentration}, one can find a hyperedge $e_1\in L_{i_1}$ for some $i_1\in \mathcal I$. Redefine $D':=D_h\setminus e_1$ and observe that $|D'|\ge |D_h| - s$. Then, by Lemma \ref{lemma:concentration}, there exists another edge $e_2\in L_{i_2}$ for some $i_2\in \mathcal I$. We repeat this argument $s(r-1)$ times and use the fact that the set $D'$ always has size at least $|D_h| - s^{2}(r-1)\ge |D_h|-s^{2}r$ to obtain the desired result.
\end{proof}

\section{Proof of Theorem \ref{thm:main}}

In this section we prove Theorem \ref{thm:main}.

\begin{proof}[of Theorem \ref{thm:main}] Actually, we have already described more or less the whole proof in Section \ref{sec:emb} but for completeness, we will repeat the main steps of the argument (for more details, the reader can read the first part of Section \ref{sec:emb} again).

Let $H\in \mathcal H((1-\varepsilon)n,\Delta)$ and let $(H',\mathcal S_1,\ldots, \mathcal S_k)$ be an $\alpha$-good partition of $H$ as obtained by Lemma \ref{lemma:Partitioning} with $\alpha$ as given in \eqref{def_alpha} (recall that $k:=k(\Delta,\alpha)$ is a constant). Let us fix $p=\beta(n)(n^{-1}\log^{1/\Delta}n)^{2/(\Delta+1)}$, where $\beta(n)\rightarrow \infty$ arbitrarily slowly whenever $n$ goes to infinity. Note that as ``containing $H$" is a monotone increasing property, it is enough to prove the theorem for this $p$.

Let $q$ be the unique number $0<q<1$ for which
\begin{equation}
  (1-q)^{k+1}=1-p,\nonumber
  \end{equation}
   and observe that as $p=o(1)$ and $k$ is a constant, it follows that $q = (1+o(1)) p/k = \Omega(p)$.

  We embed $H$ in $k+1$ rounds, where in each round $0\leq h\leq k$ we expose the edges of a graph $G_h:=G_{n,q}$, independently of the previous exposures, and extend a current partial embedding $f$ of $H_h:=H'\cup \left(\cup_{S\in \mathcal S_i,i\leq h}S\right)$, using the edges of $G_h$. Clearly, $\cup_{h=0}^{k} G_h$ has the same distribution as $G_{n,p}$. Formally, we can describe the scheme as follows:

  {\bf Round $0$.} Embed $H_0:=H'$ into $G_0$. In order to do so, observe that by property (vi) of the definition of $\alpha$-goodness, using the notation in Theorem \ref{riordan}, we have that
  $$\gamma(H')\leq (\Delta+1)/2.$$
  Therefore, using Theorem \ref{riordan} we find (whp) a copy of $H'$ in $G_0$.

  {\bf Round $h+1$.} Extend the embedding of $H_h$ into an embedding of $H_{h+1}:=H_h\cup \left(\cup_{S\in \mathcal S_{h+1}}S\right)$. This is obtained whp directly from Lemma \ref{key lemma}.

  By the union bound over $0\le h\le k$, the whole scheme succeeds whp.
\end{proof}

\emph{Acknowledgements.}  The authors would like to thank the anonymous referree for the many helpful comments and several insightful suggestions.

\bibliographystyle{abbrv}
\bibliography{hamiltonian}

\end{document}